\documentclass{amsart}
\usepackage{amsmath}
\usepackage{amsfonts}
\usepackage{amssymb}
\usepackage{graphicx}
\usepackage[mathscr]{eucal}
\newtheorem{theorem}{Theorem}
\theoremstyle{plain}

\newtheorem{lemma}{Lemma}

\numberwithin{equation}{section}
\begin{document}
\title[Generalized asymptotic Euler's relation]{Generalized asymptotic Euler's relation \\ for certain families of polytopes}
\author{L\'{a}szl\'{o} Major}\address[]{L\'{a}szl\'{o} Major \newline\indent Institute of Mathematics \newline\indent Tampere University of Technology \newline\indent PL 553, 33101 Tampere, Finland}\email[]{laszlo.major@tut.fi}

\date{Jul 6, 2011}
\keywords{prism, pyramid, stacked polytope, cyclic polytope, simplicial polytope, face lattice, $f$-vector, Euler's formula}

\begin{abstract}
According to Euler's relation any polytope $P$ has as many faces of even dimension as it has faces of odd dimension. As a generalization of this fact one can compare the number of faces whose dimension is congruent to $i$ modulo $m$ with the number of all faces of $P$ for some positive integer $m$ and for some $1\leq i \leq m$. We show some classes of polytopes for which the above proportion is asymptotically \vspace{-1.7mm} equal to $1/m$.
\end{abstract}
\maketitle
\section*{Introduction}
For any $d$-polytope $P$, we denote its $f$-vector by $f(P)$ whose $i^{th}$ component $f_i(P)$ is the number
of faces of dimension $i$ in $P$ for $i=-1,0,\ldots,d$. We put $f_{-1}(P)=f_d(P)=1$ for the improper faces $\emptyset$ and $P$ itself and $f_i(P)=0$, if $i > d$ or $i<-1$. 
A widely studied and important problem is the characterization of the set of all possible $f$-vectors of polytopes. A well-known necessary condition for an arbitrary $(d+2)$-tuple of integers to be the $f$-vector of some $d$-polytope is  Euler's relation. We give this relation in an unusual form in order to emphasize the main idea of this\vspace{1mm} paper. For any  $d$-polytope $P$ one has the equation
\begin{equation*}\label{eq1}
\sum_{i\equiv \,0 \hspace{-2mm}\mod \hspace{-0mm}2}\hspace{-2mm}f_i(P)=\sum_{i\equiv 1 \hspace{-2mm}\mod \hspace{-0mm}2}\hspace{-2mm}\vspace{1mm}f_i(P),
\end{equation*} 
in other words the number of  faces of even dimension in $P$ is equal to the number of  faces of odd  dimension. In general, for any positive integer $m$ and for any $d$-polytope $P$, we define its \emph{$m$-modular $f$-vector}  $f^{m}(P)$ to be the vector with $i^{th}$ component $f_i^m(P)$ equal to the number of faces whose dimension is congruent to $i$ modulo $m$ ($i=-1,0,\ldots,m-2$). In the special case $m=d+2$ the $m$-modular $f$-vector of $P$ is equal to the $f$-vector of $P$ and for $m=1$ we obtain the number of all faces of $P$. Using this notation, Euler's relation can be given in the following simple fo\vspace{-1mm}rm:
\begin{equation*}\label{eq2}
f_{-1}^{\hspace{0.41mm}2}(P)=f_0^{\hspace{0.41mm}2}(P)\vspace{1mm}.
\end{equation*}
The question necessarily arises from the above whether there exists some analogous relation for $m > 2$. Of course equality does not hold universally if $m > 2$. For example  the $3$-modular $f$-vector of the three dimensional cube is $(7, 9, 12)$, that is $f_{i}^{\hspace{0.41mm}3}(C)\neq f_j^{\hspace{0.41mm}3}(C)$ for $i,j\in \{-1,0,1\}$. Nevertheless the components of the $3$-modular $f$-vector of the $6$-simplex are approximately equal: $(22,21,21)$. Therefore it is reasonable to investigate the  classes of polytopes $\mathcal{P}$ for which the following statement holds. \emph{For any positive integer $m$ and for any $\varepsilon >0$ there exists a positive $K$ such that for all $P\in \mathcal{P}$ of dimension at least $K$ }

\begin{equation*}\label{eq3}
||f^m(P)/f^1(P)-( 1/m,\ldots, 1/m)\vspace{2mm}||<\varepsilon.
\end{equation*}
The aim of this paper is to present such classes of polytopes. The above statement probably holds for other natural classes of polytopes e.g. ordinary polytopes, cubical polytopes, possibly even for the class of all polytopes.

In order to be able to deal more efficiently with the $m$-modular $f$-vectors of polytopes we need the concept and some useful properties of \emph{doubly stochastic matrices}. A matrix is called \emph{positive (nonnegative)} if all  its entries are  positive (nonnegative).
A $m \times m$ nonnegative matrix $M$ is said to be \emph{doubly stochastic} if the sum of the entries in each row and in each column equals $1$. The product of doubly stochastic matrices is doubly stochastic. The following lemma can be interpreted as a special case of the main theorem of Baik and Bang  \cite{baba}. Here we provide a shorter direct proof by using general results on eigenvectors of positive matrices.
\begin{lemma}\label{lem1}Let $Q$ denote the (unique) doubly stochastic matrix $(a_{i,j})$ of order $m$ for which $a_{i,i+1}=1$ for $1\leq i \leq m-1$. If $J$ denotes the doubly stochastic matrix with all entries equal to $1/m$ and $I$ is the identity matrix of order $m$, then
\begin{equation*}\label{eq30}
\lim_{k\rightarrow \infty}\Big(\frac{aI+Q}{a+1}\Big)^k=J,\hspace{3mm}\text{where}\;a>0.
\end{equation*}
\end{lemma}
\begin{proof}First we show that if $M$ is a positive doubly stochastic matrix of order $m$, then
\begin{equation}\label{eq300}
\lim_{n\rightarrow \infty}M^n=J.
\end{equation} 
This fact can be deduced from Perron's theorem (see e.g. Theorem 8.2.8 of Horn and Johnson \cite{horn}). Namely, for any $m\times m$ positive matrix $A$ we have 
$$\displaystyle\lim_{n\rightarrow \infty }(\varrho(A)^{-1}A)^n=xy^T,$$
where $\varrho (A)$ is the spectral radius bounded as follows (see 8.1.22 of \cite{horn}):
\begin{equation}\label{111}\min_{i} \sum_{j} a_{ij} \le \varrho (A) \le \max_i \sum_{j} a_{ij}.\end{equation}
Furthermore, $x$ is a positive eigenvector of $A$ and $y$ is a positive eigenvector of  $A^T$, both corresponding to the eigenvalue $\varrho (A)$, such that $x^Ty=1.$ We obtain from (\ref{111}) that $\varrho (M)=1$, since  $\displaystyle\min_{i} \sum_{j} m_{ij} = \max_i \sum_{j} m_{ij}=1$ for the positive doubly stochastic matrix $M$.
The vectors $x=(1,1,\dotsc,1)$ and $y=(\frac 1m,\frac 1m,\dotsc,\frac 1m)$ clearly satisfy the above conditions of Perron's theorem, that is, $Mx=x$,  $M^Ty=y$ and $x^Ty=1$. Therefore we have
$\lim_{n\rightarrow \infty }M^n=xy^T=J$.
The convergence (\ref{eq300}) can also be proved directly without using Perron's theorem, see \cite{ego}.

It is easy to see that $(aI+Q)^k$ is a positive matrix if $k\geq m$ (compare to Theorem 6.2.24 of \cite{horn}), therefore the matrix $C^m=\big(\frac{aI+Q}{a+1}\big)^{m}$ is a positive doubly stochastic matrix. Hence the above result can be applied to $C^m$, that is, $\displaystyle \lim_{n \rightarrow \infty}C^{mn}=J.$ We may also express this fact by using the max norm of matrices: $$\lim_{n \rightarrow \infty}||C^{mn}-J||=0.$$  We show that for all $0\leq i\leq m$, $\displaystyle \lim_{n \rightarrow \infty}C^{(m)n+i}=J$. We recall that $JC^i=J$, so we have 
\begin{equation*}\label{eq60}
 ||C^{mn}C^i-J||=||C^{mn}C^i-JC^i||\leq ||C^{mn}-J||\cdot||C^i||.
\end{equation*} 
Because $||C^i||$ is a constant, we obtain that $\displaystyle \lim_{n \rightarrow \infty}C^{mn+i}=J$ for all $0\leq i\leq m$, consequently $C^k\rightarrow J$ as $k$ tends to infinity.
\end{proof}
The doubly stochastic matrix $Q$ in Lemma \ref{lem1} is  referred to as a \emph{cyclic permutation matrix}. In other words, if the cyclic permutation $\pi$ is defined by $\pi (v)=\pi \big((v_1,v_2,\ldots,v_n)\big)= (v_n,v_1,\ldots,v_{n-1})$ for any $v\in \mathbb{R}^n$, then  $\pi(v)$ can be given as $vQ$. Thus the vector $u=(v_n+v_1, v_1+v_2,\ldots,v_{n-1}+v_n)$ can be written as follows
\begin{equation}\label{eqer}
 u=vQ+v=v(Q+I).
\end{equation} 
This fact is used in the proofs of Theorem \ref{th1} and Theorem \ref{sp1}.
\section*{Pyramids, Bipyramids and Prisms}
Let $P_0$ be an arbitrary $(d-1)$-polytope, let $I$ be a segment not parallel to the affine hull of $P_0$. Let us recall that the vector-sum $P=P_0+I$ is said to be a \emph{prism} with base $P_0$. The convex hull of the union of $P_0$ and the segment $I$ is called a \emph{pyramid} with base $P_0$ if one endpoint of $I$ is in $P_0$ and it is called a \emph{bipyramid} with base $P_0$ if some interior point of $I$ belongs to the relative interior of $P_0.$
\begin{theorem}\label{th1}
Let $P$ be any polytope. Let us define a sequence $P_k$ of polytopes by the recursion \vspace{-2mm}
 \begin{equation}\label{eq35}\begin{split}
 \hspace{-12mm}(i)&\hspace{2mm}P_0=P\\
 \hspace{-12mm}(ii)&\hspace{2mm}P_{k+1} \hspace{1.3mm}\text{is a prism with base}\hspace{1.4mm} P_k.\\
\end{split}\end{equation}
For any positive integer $m$ 
\begin{equation*}\label{eq4}
\lim_{k\rightarrow \infty}\frac{f^m(P_k)}{f^1(P_k)}=\Big(\frac 1m,\ldots, \frac 1m\Big)\vspace{2mm}.
\end{equation*}
 \end{theorem}
\begin{proof}
We have the following relation between the $f$-vector of the prism $P_{k}$ and the $f$-vector of its base $P_{k-1}$ (\vspace{1mm}see Gr\"unbaum \cite{grun} 4.4):
\begin{equation*}\label{eq5}\begin{split}
f_{-1}(P_{k})&=2\cdot f_{-1}(P_{k-1})-1, \hspace{3mm}\\
f_0(P_{k})&=2\cdot f_0(P_{k-1})+f_{-1}(P_{k-1})-1 \hspace{3mm}\text{and}\\
f_{i+1}(P_{k})&=2\cdot f_{i+1}(P_{k-1})+f_{i}(P_{k-1})\hspace{3mm}for\hspace{2mm}\vspace{1mm} i>0.\end{split}
\end{equation*}
Applying the above we obtain directly the following equations for the components of the corresponding $m$-modular $f$-vectors:
\begin{equation}\label{eq6}\begin{split}
f&_{-1}^m(P_{k})-1=2\cdot(f_{-1}^m(P_{k-1})-1)+f_{m-2}^m(P_{k-1}), \\
f&_{0}^m(P_{k})=2\cdot f_{0}^m(P_{k-1})+f_{-1}^m(P_{k-1})-1, \\
f&_{1}^m(P_{k})=2\cdot f_{1}^m(P_{k-1})+f_{0}^m(P_{k-1}), \\
\vdots\\
f&_{m-2}^m(P_{k})=2\cdot f_{m-2}^m(P_{k-1})+f_{m-3}^m(P_{k-1}).\end{split}
\end{equation}
 Let $e_1$ denote the unit vector $(1,0,\ldots,0)$. Using the equation \ref{eqer} and the notation of Lemma \ref{lem1}, the equations (\ref{eq6}) can be formulated concisely as
 \begin{equation}\label{eq0} f^m(P_{k})-e_1=(f^m(P_{k-1})-e_1)(2I+Q),\end{equation}
 or based on $P_0$ by the recursion \ref{eq35} as
 \begin{equation}\label{eq7}
 f^m(P_{k})-e_1=(f^m(P_{0})-e_1)(2I+Q)^{k}.
 \end{equation}
 For $m=1$ the equation \ref{eq7} yields $f^1(P_{k})-e_1=(f^1(P_{0})-e_1)\cdot 3^{k}$, hence the $k^{th}$ element of the sequence of polytopes of the Theorem \ref{th1} can be written as
 \begin{equation*}\label{eq8}
\frac{f^m(P_k)}{f^1(P_k)}=\frac{(f^m(P_{0})-e_1)(2I+Q)^k+e_1}{(f^1(P_{0})-1)\cdot 3^k+1}.  
\end{equation*}
The vector $(f^m(P_{0})-e_1)/(f^1(P_{0})-1)$ is a \emph{stochastic} vector (its components sum to 1) and $(2I+Q)^k/3^k$ is a doubly stochastic matrix, therefore 
\begin{equation}\label{eq9}
\lim_{k\rightarrow \infty}\frac{f^m(P_{0})-e_1}{f^1(P_{0})-1}\cdot \Big(\frac{2I+Q}{3}\Big)^k=\Big(\frac 1m,\ldots, \frac 1m\Big)\vspace{2mm}\vspace{-0.6mm}.
\end{equation}
This fact is an immediate consequence of Lemma \ref{lem1}. The convergence \ref{eq9} implies the assertion of Theorem \ref{th1}.
\end{proof}
 It can be proved in the same way that Theorem \ref{th1} remains valid if we replace  ''prism'' with ''pyramid'' or ''bipyramid'' in condition $(ii)$. As a special case of prisms and pyramids we can mention the d-cubes and d-simplices obtained by choosing a single point as the base $P_0$, therefore for any $m$ the  components of the $m$-modular $f$-vector of the d-cube (and d-simplex) are approximately equal if $d$ is large\vspace{-0.0mm} enough.
 
\section*{Simplicial and Simple \vspace{1mm}Polytopes} 
A $d$-polytope is called \emph{simplicial} if all facets are simplices. Dually,
a $d$-polytope is called \emph{simple} if each vertex is adjacent to exactly $d$ edges. Let $\mathcal{P}_s^d$ denote the family of simple $d$-polytopes. We will deal only with simple polytopes here, however, because of the duality, similar statements are also valid for simplicial polytopes. Possibly, the most remarkable theorem related to simple (simplicial) polytopes is the $g$-theorem. This result was first conjectured by McMullen \cite{mcm1}. The sufficency of McMullens's condition is proved by Billera and Lee \cite{bile}, the necessity is proved by Stanley \cite{sta}.
We give an immediate consequence of the g-theorem here, which can be applied directly to prove Theorem \ref{sp1}: For any $S\in \mathcal{P}_s^d$ there exists an integer vector $g$ such that \begin{equation}\label{eqs5}f(S)=(f_{-1},f_0,\ldots,f_d)=gL_d-gK_d+e_1, \end{equation} 
where $e_1$ is the unit vector $(1,0,\ldots,0)$ of size $d+2$
 and $L_d$ and $K_d$ are nonnegative matrices of size $(\left\lfloor \frac d2\right\rfloor+1)\times (d+2)$ defined by
$$L_d:=\Bigg(\binom{d+2-i}{j-1}\Bigg)_{1\leq i\leq d/2+1,\, 1\leq j \leq d+2}$$ 
$$K_d:=\vspace{2mm}\Bigg(\binom{i-1}{j-1}\Bigg)_{1\leq i\leq d/2+1,\, 1\leq j \leq d+2} $$
Actually, $L_d$ and $K_d$ are constructed from certain rows of Pascal's triangle, in other words, from the $f$-vectors of certain $d$-simplices. This observation turns out to be useful to deduce the following statement from the case of simplices.
\begin{theorem}\label{sp1}For any positive integer $m$ and for any $\varepsilon >0$ there exists a positive $K$ such that for any simple polytope $S$ of dimension at least $K$  \begin{equation*}\left\|\label{eqs4}\frac{f^m(S)}{f^1(S)}-\Big(\frac 1m,\ldots, \frac 1m\Big)\right\|<\varepsilon\vspace{2mm}. \end{equation*}  \end{theorem}
\begin{proof}Let us fix a positive integer $m$ and a positive $\varepsilon$. We apply the notation and methods of the proof of Theorem \ref{th1} 
 with some modification. In addition we put $A=I+Q$ (see Lemma \ref{lem1}). For a vector $v$ let $\sigma(v)$ denote the sum of the components of $v$ and $\delta (v)$ the difference between the maximum and the minimum components of $v$. Furthermore, we shall denote the \emph{reduced spread} $\frac{\delta(v)}{\sigma(v)}$ of the vector $v$ by $\delta^*(v)$.
Using this notation, our goal is to show that the reduced spread $\delta^*(f^m(S))$ is less than $2\varepsilon$ if the dimension of $S$ is large enough.
First we consider some simple facts about the $\delta$  and $\sigma$ functions. Clearly we have  $\delta(u+v)\leq \delta(u)+\delta(u)$ and $\delta(u-v)\leq \delta(u)+\delta(v)$. The $\delta$ function is invariant under a cyclic permutation, that is, $\delta(v)=\delta(vQ)$, consequently for any nonnegative integer $n$ we have $\delta(e_1A^{n+1})\leq 2\delta(e_1A^n)$. Indeed, $\delta(e_1A^{n+1})=\delta(e_1A^{n}Q+e_1A^n)\leq 2\delta(A^n)$. Therefore by induction on $i$ one obtains that for some $1\leq i \leq n$
\begin{equation}\label{eqss1}\delta(e_1A^n)\leq 2^i\delta(e_1A^{n-i}). \end{equation}
In the similar way one can show that for some $1\leq i \leq n$
\begin{equation}\label{eqss2}\sigma(e_1A^n)= 2^i\sigma(e_1A^{n-i}). \end{equation}
Now, we define the $m-$modular versions of the generator matrices $L_d$ and $K_d$ (involved in the $g$-theorem, see above) as follows: let $L^m_d$ and $K^m_d$ be the $(\left\lfloor \frac d2\right\rfloor+1)\times m$ matrices with entries \vspace{1.7mm} respectively
$$l_{i,j}:=\hspace{-3mm}\sum_{r\equiv j-1 \hspace{-2mm}\mod m}{\binom{d+2-i}{r}}\hspace{3mm} \text{and}\hspace{3mm} k_{i,j}:=\hspace{-3mm}\sum_{r\equiv j-1 \hspace{-2mm}\mod m}\vspace{1.6mm}{\binom{i-1}{r}} $$
for $1\leq i \leq d/2+1$ and $1\leq j \leq m.$   For \vspace{1.6mm}example
\[ L_5=\begin{pmatrix}
1&6&15&20&15&6&1\\
1&5&10&10&5&1&0\\
1&4&6&4&1&0&0\\
\end{pmatrix},\quad L^4_5=\begin{pmatrix}
16&12&16&20\\
6&6&10&10\\
2&4&6&4\\
\end{pmatrix}.
\] 
By the $g$-theorem, the $m$-modular $f$-vector of a simple $d$-polytope $S$ can be given in the following form:
\begin{equation}\label{eqs6}f^m(S)=gL^m_d-gK^m_d+e_1, \end{equation} for some integer vector $g=(g_1,\ldots,g_{\left\lfloor \frac d2\right\rfloor+1})$.
The rows of the matrices $L^m_d$ and $K^m_d$ can be given by using certain powers of the matrix $A$ (as in equation \ref{eq7}). Namely, $e_kL^m_d=e_1A^{d-k+2}$ and $e_kK^m_d=e_1A^{k-1}$, where $e_k$ is the $k^{th}$ unit vector (of the required size). Therefore we have
\begin{equation}\label{eqs7}f^m(S)=e_1+\sum_{k=1}^{\left\lfloor \frac d2\right\rfloor+1} g_ke_1(A^{d-k+2}-A^{k-1})\vspace{1.7mm}\end{equation}
Using equation \ref{eqs7} the reduced spread of $f^m(S)$ can be given as follows:
$$\delta^*(f^m(S))=
\frac{\delta\big(e_1+\sum g_ke_1(A^{d-k+2}-A^{k-1})\big)}{\sigma\big(e_1+\sum g_ke_1(A^{d-k+2}-A^{k-1})\big)}$$
For any $1\leq k \leq \lfloor \frac d2\rfloor+1$ we have
 $\delta\big(e_1(A^{d-k+2}-A^{k-1})\big)\leq \delta\big(e_1A^{d-k+2}\big)+\delta\big(e_1A^{k-1}\big)$ and $\sigma\big(e_1(A^{d-k+2}-A^{k-1})\big)= \sigma\big(e_1A^{d-k+2}\big)-\sigma\big(e_1A^{k-1}\big)$. Therefore we can bound the reduced spread $f^m(S)$ from above: 
 $$\delta^*(f^m(S))\leq 
 \frac{1+\sum g_k\cdot\big(\delta (e_1A^{d-k+2})+\delta (e_1A^{k-1})\big)}{1+\sum g_k\cdot\big(\sigma (e_1A^{d-k+2})-\sigma (e_1A^{k-1})\big)}$$
 For any $1\leq k \leq \lfloor \frac d2\rfloor+1$ applying the inequalities \ref{eqss1} and \ref{eqss2} we obtain that
 \[\begin{split} 2\delta (e_1A^{d-k+2})&\geq \delta (e_1A^{d-k+2})+\delta (e_1A^{k-1})\hspace{2mm}\text{and}\\
 \frac12\sigma (e_1A^{d-k+2})&\leq\sigma (e_1A^{d-k+2})-\sigma (e_1A^{k-1}).
  \end{split}\]
Therefore the reduced spread can be bounded as follows:
$$\delta^*(f^m(S))\leq \frac{1+\sum g_k  2\delta (e_1A^{d-k+2})}{\sum g_k\frac12\sigma (e_1A^{d-k+2})}\leq  6\cdot\frac{\sum g_k\delta (e_1A^{d-k+2})}{\sum g_k\sigma (e_1A^{d-k+2})}$$ 
Using again the inequalities \ref{eqss1} and \ref{eqss2} we have for any $1\leq k \leq \lfloor \frac d2\rfloor+1$
 \[\begin{split}\delta (e_1A^{d-k+2})&\leq 2^{\left\lfloor \frac d2\right\rfloor-k+1}\delta (e_1A^{d-\left\lfloor \frac d2\right\rfloor+1}) \hspace{2mm}\text{and}\\
 \sigma (e_1A^{d-k+2})&= 2^{\left\lfloor \frac d2\right\rfloor-k+1}\sigma (e_1A^{d-\left\lfloor \frac d2\right\rfloor+1}).
  \end{split}\]
Hence the reduced spread $\delta^*(f^m(S))$ is less than or equal to
$$ 6\cdot\frac{\delta (e_1A^{d-\left\lfloor \frac d2\right\rfloor+1})\cdot\sum g_k\cdot 2^{\left\lfloor \frac d2\right\rfloor-k+1}}{\sigma (e_1A^{d-\left\lfloor \frac d2\right\rfloor+1})\cdot\sum g_k \cdot 2^{\left\lfloor \frac d2\right\rfloor-k+1}}=6\cdot\delta^*(e_1A^{d-\left\lfloor \frac d2\right\rfloor+1}) $$
 where $e_1A^{d-\left\lfloor \frac d2\right\rfloor+1}$ is the $m$-modular $f$-vector of a $(d-\left\lfloor \frac d2\right\rfloor)-$dimensional simplex. By Lemma \ref{lem1}, the reduced spread $\delta^*(e_1A^{d-\left\lfloor \frac d2\right\rfloor+1})$ tends to zero as $d-\left\lfloor \frac d2\right\rfloor+1$ tends to infinity. Consequently, for $\frac 13  \varepsilon$ there exists a positive $K'$ such that $\delta^*(e_1A^{d-\left\lfloor \frac d2\right\rfloor+1})<\frac 13 \varepsilon$ if $d-\left\lfloor \frac d2\right\rfloor+1>K'$. We put $K=2K'$. If $d>K$, then $\delta^*(e_1A^{d-\left\lfloor \frac d2\right\rfloor+1})<\frac 13 \varepsilon$, therefore $\delta^*(f^m(S))<2\varepsilon.$
\end{proof}

\end{document}